\documentclass{amsart}
\usepackage{amssymb}
\usepackage{amsmath}
\usepackage{url}
\usepackage{color}
\usepackage{setspace}
\usepackage{enumerate}

\usepackage{setspace}
\usepackage{epsfig}
\usepackage{graphicx}
\usepackage{psfrag}

\newtheorem{theorem}{Theorem}

\newtheorem{proposition}[theorem]{Proposition}

\newtheorem{fact}[theorem]{Fact}
\newtheorem{remark}[theorem]{Remark}
\newtheorem{definition}[theorem]{Definition}

\newtheorem{observation}[theorem]{Observation}

\newcommand{\ignore}[1]{}
\newcommand{\tr}{\text{tr}}

\newcounter{notes}
\begin{document}

\numberwithin{theorem}{section}
\numberwithin{equation}{section}

\title[Lengthening deformations of singular hyperbolic tori]{Lengthening deformations \\ of singular hyperbolic tori}
\author[F.\ Gu\'eritaud]{Fran\c{c}ois Gu\'eritaud}
\address{CNRS and Universit\'e Lille 1, Laboratoire Paul Painlev\'e, 59655 Villeneuve d'Ascq Cedex, France 
\newline Wolfgang-Pauli Institute, University of Vienna, CNRS-UMI 2842, Austria}\email{francois.gueritaud@math.univ-lille1.fr}

\begin{abstract} Let $S$ be a torus with a hyperbolic metric admitting one puncture or cone singularity. We describe which infinitesimal deformations of $S$ lengthen (or shrink) all closed geodesics. We also study how the answer degenerates when $S$ becomes Euclidean, i.e.\ very small. \end{abstract}

\maketitle
\begin{center}
\begin{minipage}{0.8\textwidth}
{\footnotesize \begin{spacing}{1.0} \noindent {\sc R\'esum\'e.} Soit $S$ un tore muni d'une m\'etrique hyperbolique admettant un trou ou une singularit\'e conique. Nous d\'ecrivons quelles d\'eformations infinit\'esimales de $S$ allongent (ou raccourcissent) toutes les g\'eod\'esiques ferm\'ees. Nous \'etudions aussi comment la r\'eponse \`a cette question d\'eg\'en\`ere lorsque $S$ devient euclidienne, c'est-\`a-dire tr\`es petite.\end{spacing}}
\end{minipage}
\end{center}

\section{Introduction}

\subsection{Context} \label{context}
Given a hyperbolic surface $(S,g)$ with geodesic boundary, one can ask which infinitesimal deformations of $g$ in Teichm\"uller space have the property of \emph{lengthening all closed geodesics} of $g$. A simple surgery argument (see Lemma 3.4 of \cite{stretchmaps}) shows that if a self-intersecting closed curve is shortened, then so is another closed curve with fewer self-intersections; so the above question can be restricted to \emph{simple closed curves} only. A slightly stronger requirement is for the deformation to lengthen all \emph{measured laminations}. This choice is more natural since measured laminations, when seen up to scalar multiplication, form a \emph{compact topological manifold} in which simple closed curves lie densely (in fact, the surgery argument extends, from closed vs.\ simple closed geodesics, to currents vs.\ measured laminations). An infinitesimal deformation $v\in T_g \mathcal{T}_S$, where $\mathcal{T}_S$ is the Teichm\"uller space of $S$, will be called \emph{lengthening} if it 
lengthens all measured laminations to first order. Drumm proved \cite{drumm} that lengthening deformations exist whenever $\partial S \neq \emptyset$.

Associated to $v\in T_g \mathcal{T}_S$ is an affine representation $$\rho_v:\pi_1(S) \rightarrow \text{SO}_{2,1}(\mathbb{R}) \ltimes_{\text{Ad}} \mathfrak{so}_{2,1}(\mathbb{R})\simeq \text{Isom}^+(\mathfrak{so}_{2,1}(\mathbb{R}))$$ whose projection to $\text{SO}_{2,1}(\mathbb{R})$ coincides with (the holonomy representation of) $g$. It was Margulis \cite{margulis} who first noticed that 
$\rho_v$ can define a \emph{properly discontinuous} action on $\mathfrak{so}_{2,1}(\mathbb{R}) \simeq \mathbb{R}^{2+1}$, thus exhibiting the first examples of properly discontinuous affine actions on $\mathbb{R}^n$ by a nonabelian free (or even non-polycyclic) group. In our terminology, his construction was to study $\rho_v$ for an appropriate, lengthening $v$. In \cite{golama}, Goldman, Labourie and Margulis refined and generalized these results by showing that among the space of $g$-compatible translational parts in $\text{Isom}^+(\mathfrak{so}_{2,1}(\mathbb{R}))$, or $g$-cocycles, those which define a properly discontinuous action \emph{coincide} (via $v \mapsto \rho_v$) with the convex open cone $C$ of lengthening deformations of $g$ in $T_g\mathcal{T}_S$. They prove this under the technical restriction that $S$ should have no parabolics, but this is likely irrelevant. See also \cite{dgk}.

In general the boundary of this open cone $C$ has many facets of varying dimensions, which one would like to relate to the geometry of the surface $S$, for example the geometry of its curve complex. Roughly speaking, to each point on $\partial C$ should correspond some curve or lamination which fails to be lengthened. This subject is still quite open; see however \cite{dgk2}.

\subsection{Main results}

For $S$ a pair of pants (hence $\dim \mathcal{T}_S=3$), one can show that $C$ is a cone on a triangle, and Charette, Drumm and Goldman \cite{3holed} even provided fundamental domains for the corresponding actions on $\mathbb{R}^{2+1}$. This was generalized in \cite{2holed} to $S$ a two-holed projective plane, and more recently in \cite{dgk2} to $S$ a surface of arbitrary complexity. The convex cone $C$ then has an intricate structure, related to the arc complex of $S$.

But let us back up. After the 3-holed sphere, the next (orientable) case, to which we will restrict now, is when $S$ is a one-holed torus: one still has $\dim \mathcal{T}_S=3$, so $C$ is a cone over a certain projective, planar, convex region $\Pi$. The boundary $\partial \Pi$ of $\Pi$ will turn out to contain many ``sides'', i.e.\ maximal segments not reduced to points: so we will refer to $\Pi$ as a ``polygon''. The term is somewhat abusive because the number of sides is infinite, but it is relevant in that we care about enumerating these sides. Namely, a side corresponds to the set of infinitesimal deformations that lengthen all simple closed curves except one --- this exceptional curve is then naturally associated to the side. As early as 2003, Charette \cite{charette} proved that the simple closed curves corresponding to the sides of $\Pi$ can have arbitrarily long expressions in the generators of $\pi_1(S)$. Charette, Drumm and Goldman prove in \cite{1holed} that in fact, all simple closed curves arise.
 Our purpose in this note is:

\begin{enumerate}[(A)]
\item \label{item1} to prove that for every hyperbolic metric $g$ on $S$, the polygon $\Pi$ has one side associated to each simple closed curve in $S$ (this recovers one of the main results of \cite{1holed}, and also follows from a special case of \cite{dgk2}: see \S 6 therein);
\item \label{item2} to extend this result to the case when the boundary component of $S$ is replaced with a cone singularity;
\item \label{item3} to provide formulas and estimates for quantities such as the side lengths of the infinite polygon $\Pi$ (in appropriate affine charts);
\item \label{item4} to analyze how $\Pi$ degenerates when the area of the singular hyperbolic surface $S$ goes to $0$.
\end{enumerate}
The following, precise formulation of \eqref{item2} can be seen as our main result: let $S$ be a one-holed torus and $\mathcal{T}$ the space of isotopy classes of hyperbolic metrics on $S$ whose completion admits either a geodesic boundary component (at the hole), or a cusp, or a cone singularity whose measure lies in $(0,2\pi)$. For every simple closed (geodesic) curve $\gamma$ in $S$, let $\ell_\gamma:\mathcal{T}\rightarrow \mathbb{R}$ be the associated length function.  
\begin{theorem} \label{thm:main-intro}
The deformation space $\mathcal{T}$ is a smooth $3$-manifold. For every $g\in \mathcal{T}$, the lengthening infinitesimal deformations of $g$ in $T_g\mathcal{T}$ form a cone over an infinite polygon $\Pi$ such that for every simple closed curve $\gamma$ in $S$, the differential $\mathrm{d}\ell_\gamma$ vanishes on a nondegenerate side of $\Pi$.   
\end{theorem}

Theorem \ref{thm:main-intro} will be proved in Section \ref{sec:derivatives}. The goal \eqref{item3}  will be the object of computational sections \ref{sec:computation} and \ref{sec:asymptotics}. The goal \eqref{item4} is dealt with in Sections \ref{sec:onecurve} and \ref{sec:allcurves}. 

\subsection{Comments}

$\bullet$ Perhaps the most striking aspect of \eqref{item4} is Theorem \ref{thm:allpinch}, saying that in the limit when the cone angle of $S$ becomes $2\pi$ and the diameter of $S$ goes to $0$, the polygon $\Pi$ becomes a round disk naturally identified with the space of normalized \emph{Euclidean} metrics on $S$ (in other words, with $\mathbb{H}^2$).

\smallskip \noindent $\bullet$
Note that when $(S,g)$ has a cone singularity, a lengthening deformation $v$ of $g$ no longer defines a proper affine action on $\mathbb{R}^{2+1}$. Nevertheless, it would be nice to prove that $v$ defines a flat complete Lorentzian manifold with some kind of singularity, for an appropriate definition of Lorentzian completeness. We leave this question for future work and concentrate on describing the lengthening deformations themselves.

\smallskip \noindent $\bullet$
In general, one should not make light of the fact that lengthening deformations, as defined in Section \ref{context}, are required to lengthen all \emph{irrational measured laminations} as well as simple closed curves. To be precise, the length function $\ell_\mu$ associated to an irrational measured lamination $\mu$ will vanish on a supporting plane of $\Pi$ that contains just one point of $\partial \Pi$ (not belonging to any side). The key fact here is continuity on the space of measured laminations $\mu$: the map taking $\mu$ to the associated \emph{current} in the unit tangent bundle of $S$ is continuous, and therefore so is $\mu \mapsto \ell_\mu\in T^*\mathcal{T}$, as Thurston noted (in the absence of cone singularities) in \cite{stretchmaps}. See also \cite{dgk}. However, we will abstain from discussing such technicalities in this paper, and deliberately keep the focus just on simple closed curves.

\smallskip \noindent $\bullet$
In going from \eqref{item1} to \eqref{item2} above, it is striking that the polygon $\Pi$ does not seem to undergo any qualitative change at all. This should be related to the fact that in a hyperbolic surfaces with cone singularities, simple geodesics stay away fom the cone points, at least if the cone angles are $<\pi$. Thus one should expect such stability of the cone $C$ (or its base $\Pi$) to hold even for more general surfaces. For singularities $>\pi$ the ``combinatorics'' of $C$ might conceivably undergo abrupt changes in general, because the singularities no longer repel simple closed curves. However for $S$ a one-holed torus this does not happen, due to the hyperelliptic involution (which preserves all simple closed curves): this explains why Theorem~\ref{thm:main-intro} holds all the way to cone angle $2\pi$. 

\subsection*{Acknowledgments}
This work sprang from discussions with Bill Goldman and Virginie Charette at a conference in Autrans in January 2010, and became a kind of unpublished  precursor to \cite{dgk} and \cite{dgk2}. I am grateful to Maxime Wolff for pointing out results from \cite{goldman}. The remarks of an anonymous referee greatly contributed to bringing this paper into a readable form.
It is a pleasure to see it appear in this volume in the honor of Michel Boileau. 

This work was partially supported by grants DiscGroup (ANR-11-BS01-013) and ETTT (ANR-09-BLAN-0116-01), and through the Labex CEMPI (ANR-11-LABX-0007-01).

\section{Markoff maps}\label{sec:markoff}

\subsection{Definitions}

Our main tool to describe metrics $g$ on the one-holed torus $S$ will be \emph{Markoff maps}. Markoff maps were invented by Bowditch in \cite{bowditch} as a fine bookkeeping device for maps from the free group on two generators to $\text{SL}_2(\mathbb{K})$, where $\mathbb{K}$ is a field. We briefly recall the definition below. It will rely on some facts from the very classical and beautiful theory of \emph{Farey sequences}, or ``combinatorics of~$\mathbb{Q}$'' (\cite[\S 3]{hardy}, \cite{ford}): in recent literature, see \cite[pp. 152--156]{bookofnumbers} for an introduction and \cite[\S 8]{bonahon} for the link with the one-holed torus $S$.

Endow $S$ with a marking, i.e.\ a homeomorphism to $(\mathbb{R}^2\smallsetminus \mathbb{Z}^2)/\mathbb{Z}^2$, so that every (unoriented, non-boundary-parallel) simple closed curve in $S$ is characterized by its \emph{slope}, a number $\frac{p}{q}\in \mathbb{P}^1\mathbb{Q}\simeq \mathbb{Q}\cup\{\infty\}$. Identify the boundary at infinity of $\mathbb{H}^2$ with the real projective line via the upper half plane model, and recall the \emph{Farey triangulation} of $\mathbb{H}^2$: rational points (in reduced form) $\frac{p}{q}$ and $\frac{p'}{q'}$ of $\mathbb{P}^1\mathbb{Q}\subset \partial_{\infty}\mathbb{H}^2$ are connected by a line of $\mathbb{H}^2$ if and only if $|pq'-qp'|=1$. This results in the ideal triangle $01\infty$ together with all its iterates (called Farey triangles) under reflections in its sides. This Farey triangulation is invariant under the group $\text{SL}_2(\mathbb{Z})$, or its quotient $\text{PSL}_2(\mathbb{Z})$, acting by M\"obius transformations on the upper half-plane of $\mathbb{C}$. 
This action is transitive on oriented Farey edges; in fact the element of $\mathrm{PSL}_2(\mathbb{Z})$ taking the edge $(\infty, 0)$ to $(\frac{p}{q},\frac{p'}{q'})$ is just $(\begin{smallmatrix}p&p'\\q&q'\end{smallmatrix})$ or $(\begin{smallmatrix}-p&p'\\-q&q'\end{smallmatrix})$. We say that $\frac{p}{q}$ and $\frac{p'}{q'}$ as above 
are \emph{Farey neighbors}. The group $\text{SL}_2(\mathbb{Z})$ also realizes the (orientation preserving) mapping class group by its action on $S=(\mathbb{R}^2\smallsetminus \mathbb{Z}^2)/\mathbb{Z}^2$; these two actions are compatible, via the ``slope'' function defined above.

Consider the infinite, trivalent \emph{Farey tree} $\Gamma\subset \mathbb{H}^2$ whose vertices are the centers of the Farey triangles, and whose edges connect the centers of adjacent triangles. Regions of the complement of $\Gamma$ are indexed by rational numbers $\frac{p}{q}$: let us denote by $R_{\frac{p}{q}}$ the region whose closure intersects the real projective line precisely at the point $\frac{p}{q}$, and by $\mathcal{R}$ the collection of all such regions as $\frac{p}{q}$ ranges over $\mathbb{P}^1\mathbb{Q}$.

\begin{definition}(From \cite{bowditch}.)
A \emph{Markoff map} is a map $$\Phi:\mathcal{R}\rightarrow \mathbb{R}$$ such that whenever $R,R',R'', R'''$ are distinct regions such that each pair among them is adjacent except $\{R,R'''\}$, one has 
\begin{equation} \Phi(R) + \Phi(R''') = \Phi(R')\Phi(R'')~.  \label{eq:markoff} \end{equation}
\end{definition}
In such a configuration there always exist two Farey neighbors $\frac{p}{q}, \frac{p'}{q'}$ such that $$(R,R',R'',R''')=\left ( R_{\frac{p+p'}{q+q'}} , R_{\frac{p}{q}} , R_{\frac{p'}{q'}} , R_{\frac{p-p'}{q-q'}} \right )~.$$
(In particular, $R$ is exchanged with $R'''$ if we replace the coprime integers $p',q'$ with $-p',-q'$.) In the sequel, we will typically represent Markoff maps $\Phi$ in the plane by finite trivalent subtrees of the Farey tree, together with numbers $\Phi(R)$ written in some of the regions $R=R_{\frac{p}{q}}$.

\subsection{Representations determine Markoff maps} \label{repmar} For any $\underline{A},\underline{B} \in \text{SL}_2\mathbb{R}$, 
$$\mathrm{tr}(\underline{A}\underline{B})+\mathrm{tr}(\underline{A}\underline{B}^{-1})=\mathrm{tr}(\underline{A})\mathrm{tr}(\underline{B})$$ 
(where $\mathrm{tr}$ is the trace operator),
because $\underline{B}+\underline{B}^{-1}=\mathrm{tr}(\underline{B})\cdot \text{Id}$ and $\mathrm{tr}$ is linear. This identity is formally the same as (\ref{eq:markoff}). In fact, suppose $\frac{p}{q}$ and $\frac{p'}{q'}$ are Farey neighbors, and suppose $\underline{A}$ and $\underline{B}$ are the holonomies (for the metric $g$ and a choice of basepoint) of simple closed curves of slopes $\frac{p}{q}$ and $\frac{p'}{q'}$ in the one-holed torus $S$. Then $\underline{A}\underline{B}$ and $\underline{A}\underline{B}^{-1}$ are (up to order) the holonomies of curves of slopes $\frac{p+p'}{q+q'}$ and $\frac{p-p'}{q-q'}$. Therefore, \emph{the metric $g\in\mathcal{T}$ defines a Markoff map} $\Phi=\Phi_g$, by associating to each region $R_{\frac{p}{q}}$ the trace, under the holonomy representation of $g$, of an element of $\pi_1(S)$ representing the simple closed curve of slope $\frac{p}{q}$.

Since holonomies live in $\mathrm{PSL}_2\mathbb{R}$, the possible lifts to $\mathrm{SL}_2\mathbb{R}$ actually define $4$ Markoff maps, equal up to sign. It follows from \cite[\S 2]{goldman} that for $g\in\mathcal{T}$, one of these Markoff maps takes only values $>2$: this is the one we shall refer to as $\Phi_g$.

\subsection{Trace identity}
\label{sec-traceid}
The trace identity
\begin{equation} \label{traceid}
\mathrm{tr}[\underline{A},\underline{B}]=\mathrm{tr}^2 \underline{A} + \mathrm{tr}^2 \underline{B} + \mathrm{tr}^2 (\underline{AB}) - \mathrm{tr} \underline{A} \mathrm{tr} \underline{B} \mathrm{tr} (\underline{AB}) -2
\end{equation}
holds for all $\underline{A},\underline{B}\in\mathrm{SL}_2\mathbb{R}$ (note that it is a well-defined real even if $\underline{A},\underline{B}$ are only given in $\mathrm{PSL}_2\mathbb{R}$). It will play an important role in this paper, because for $\alpha,\beta$ the standard generators of $\pi_1 S \simeq \mathbb{F}_2$, the commutator $[\alpha,\beta]=\alpha\beta \alpha^{-1}\beta^{-1}$ is the loop around the hole of $S$: thus its trace under the holonomy representation of $g$ (taking $\alpha$ to $\underline{A}$ and $\beta$ to $\underline{B}$) must be the above number, and belong to $(-2, 2)$ in the case of a conical singularity, or to $(-\infty, -2]$ in the case of a cusp or funnel (see \cite[\S 3]{goldman}).

\subsection{Markoff maps determine representations} 
\label{sec-marrep}
Equation (\ref{eq:markoff}) clearly shows that a Markoff map is completely determined by the values it takes on any three pairwise adjacent regions. It is also a classical, elementary result that in $\mathrm{SL_2\mathbb{R}}$, the traces of elements $\underline{A}, \underline{B}$ and $\underline{AB}$ determine $\underline{A}$ and $\underline{B}$ up to conjugation --- unless
$\mathrm{tr}[\underline{A},\underline{B}]$ (computed via \eqref{traceid}) is equal to $2$. (In the latter case, $\underline{A}$ and $\underline{B}$ can be simultaneously conjugated to upper triangular matrices, but only their diagonal part is determined by the traces; the associated diagonal representation factors through the abelianization $\mathrm{Ab}(\pi_1 S)\simeq \mathbb{Z}^2$.)

Up to this ambiguity when $\mathrm{tr}[\underline{A},\underline{B}]=2$, we can thus \emph{identify} conjugacy classes of representations with (certain) Markoff maps, by \S \ref{repmar}--\ref{sec-marrep}.

\subsection{The Markoff maps of hyperbolic metrics $g$} \label{curlix}
\begin{fact} \label{factii} (Contained in \cite[\S 3]{goldman}).
The Markoff maps $\Phi_g$ associated to metrics $g\in\mathcal{T}$ are exactly the Markoff maps $\Phi:\mathcal{R}\rightarrow \mathbb{R}$ satisfying
\begin{equation}\label{eqii}
\left \{
\begin{array}{rrcll}
\mathrm{(i)} & \Phi(R_{\frac{p}{q}}) & >& 2 & \text{for all coprime $p,q\in\mathbb{Z}$;} \\
\mathrm{(ii)} & K:= \mathrm{tr}\, \rho([\alpha, \beta])& <& 2.& 
\end{array}
\right .
\end{equation}
Such Markoff maps are called \emph{geometric}.
\end{fact}
Note that \eqref{traceid} lets us compute $K$ of (\ref{eqii}.ii) from
$$
\begin{array}{lclclcl}
A&:=& \Phi(R_\infty)&=&\mathrm{tr}\,\rho(\alpha) &=& \mathrm{tr}\, \underline{A}\\                                 
B&:=& \Phi(R_0)&=&\mathrm{tr}\,\rho(\beta) &=&\mathrm{tr}\,\underline{B}\\   
C&:=& \Phi(R_1)&=&\mathrm{tr}\,\rho(\alpha\beta) &=& \mathrm{tr}\, \underline{AB},
\end{array}
$$
namely $K=A^2+B^2+C^2-ABC-2$. More generally this formula holds whenever $A,B,C$ are the values of $\Phi$ on three neighboring regions.
We will also be interested in the degenerations as $\Phi(R_{\frac{p}{q}})\rightarrow 2$ or $K\rightarrow 2$, i.e.\ the limit cases of \eqref{eqii}.

\begin{remark} \label{rem:linked}
Conditions (i) and (ii) above are highly overlapping in content. In fact, it is easy to see by \eqref{eq:markoff} and induction on the regions that (ii) implies (i) under the sole extra assumption that $A, B$ are $\geq 2$ (i.e.\ taking correct lifts to $\mathrm{SL}_2(\mathbb{R}))$.
In the opposite direction, (i) is known to imply $K\leq 2$ by \cite[Th.\ 5.2.1]{goldman}. 
\end{remark}

By Fact \ref{factii}, the deformation space $\mathcal{T}$ of isotopy classes of complete hyperbolic metrics $g$ on the once punctured torus, possibly with a cone singularity at the puncture, identifies with the space $\mathcal{X}$ of geometric Markoff maps $\Phi_g$ (hence the name ``geometric'').

Using $A, B, C$ as coordinates, the space $\mathcal{X}$ is an open subset of the octant $\{A,B,C\geq 2\}$ of $\mathbb{R}^3$, tangent to its three faces along their bisecting rays $(2,t,t)_{t\geq 2}$, $(t,2,t)_{t\geq 2}$, $(t,t,2)_{t\geq 2}$, and containing the round cone tangent at these same rays. Near $(2,2,2)$, the open set $\mathcal{X}$ looks to first order like the round cone, but far away the set of asymptotic directions of $\mathcal{X}$ is the full octant: see Figure \ref{fig:wedge}.

\begin{figure}[h!] \centering
\includegraphics[width=3.5cm]{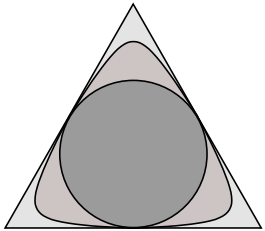}
\caption{Three slices of $\mathcal{X}$ at different values of $A+B+C-6$ (outwards: very small; $1$; very large) --- all scaled to fit into the same triangle.} \label{fig:wedge} 
\end{figure}

\subsection{Properties}
Let $\Phi>2$ be a geometric Markoff map, i.e.\ satisfy \eqref{eqii}.
Consider a region $R$ together with its ordered sequence of neighbors: up to changing the marking, $R=R_{\infty}$ and the neighbors are $(R_n)_{n\in \mathbb{Z}}$. Equation (\ref{eq:markoff}) yields $\Phi(R_{n+1})=\Phi(R_{\infty})\Phi(R_n)-\Phi(R_{n-1})$. Therefore, if $a>1$ is such that 
$$a+a^{-1}=A = \Phi(R_{\infty}),$$ 
then 
\begin{equation} \label{eq:binet}
\Phi(R_n)=\lambda a^n + \mu a^{-n}
\end{equation}
for some $\lambda, \mu$ independent of $n$. 
Using the $\cosh$ function to represent sums of exponentials such as \eqref{eq:binet}, this can be summarized as follows:

\begin{remark}
For $\Phi$ a geometric Markoff map, if $\Phi(R_{\infty})=2\cosh \ell$ then there exist $y>0$ and $x\in \mathbb{R}$ such that for all $n\in\mathbb{Z}$,
 $$\Phi(R_n)=2y\cdot\cosh(n\ell-x).$$
\end{remark}

\begin{proposition} \label{prop:commutator}
For a geometric Markoff map, 
$y>1$.
\end{proposition}
\begin{proof}
Recall $y>0$. 
When $\tr \, \rho(\alpha)= 2\cosh \ell$ and $\tr \, \rho(\beta)= 2y \cdot \cosh (-x)$ and $\tr \, \rho(\alpha \beta)= 2y \cdot \cosh (\ell-x)$, the trace identity \eqref{traceid} reads
$$\tr \, \rho([\alpha,\beta])=2+4(1-y^2)\sinh^2 \ell.$$ By (\ref{eqii}.ii) we want this number to stay $<2$, hence $y>1$. 
\end{proof}
Note that the case $y=1$ corresponds to the area of the singular torus collapsing to $0$. The case $y=\coth \ell >1$ corresponds to $\tr \, \rho([\alpha, \beta])=-2$, a cusp (parabolic commutator with nonzero Euler class).

\section{Derivatives of Markoff maps} \label{sec:derivatives}

In this section we prove Theorem \ref{thm:main-intro}. Let $\Phi:\mathcal{R}\rightarrow \mathbb{R}$ be a Markoff map satisfying \eqref{eqii}, and let us focus on infinitesimal deformations of $\Phi$ (or of the corresponding singular hyperbolic metric $g\in\mathcal{T}$ on the one-holed torus).

If we let $(\Phi_t)_{t\geq 0}$ be a smooth deformation of $\Phi$ that lengthens all curves of slope different from $\infty$ but leaves $\Phi_t(R_{\infty})$ constant, then $\frac{\text{d}}{\text{d}t}_{|t=0}\Phi_t(R_n)\geq 0$ for all $n\in \mathbb{Z}$. By considering $n\rightarrow \pm\infty$ in (\ref{eq:binet}), this is equivalent to $\frac{\text{d}\lambda}{\text{d}t}\geq 0$ and $\frac{\text{d}\mu}{\text{d}t}\geq 0$. Therefore, the directions of deformation which increase all $\Phi(R_n)$ but fix $\Phi(R_{\infty})$, seen as projective triples $$\left [ \frac{\text{d}}{\text{d}t}\Phi_t(R_{\infty})~:~ \frac{\text{d}}{\text{d}t}\Phi_t(R_0) ~:~ \frac{\text{d}}{\text{d}t}\Phi_t(R_1) \right ]~,$$ form the projective segment $\boldsymbol{\alpha}$ from $[0:1:a]$ to $[0:a:1]$, containing $[0:1:1]$. 
\begin{remark} \label{rem:geoseq}
The two endpoints of $\boldsymbol{\alpha}$ are characterized by the fact that 
\begin{itemize} 
\item $\frac{\mathrm{d}}{\mathrm{d}t} \Phi_t(R_{\infty})=0$, and 
\item $\left ( \frac{\mathrm{d}}{\mathrm{d}t} \Phi_t(R_n)\right )_{n\in \mathbb{Z}}$ is a geometric sequence (necessarily of ratio $a$ or $a^{-1}$).
\end{itemize}
\end{remark}

\begin{figure}[h!] \centering
\psfrag{a}{$a$}
\psfrag{b}{$b$}
\psfrag{c}{$c$}
\psfrag{0}{$0$}
\psfrag{1}{$1$}
\psfrag{A}{$A'$}
\psfrag{B}{$B'$}
\psfrag{C}{$C'$}
\psfrag{ga}{$\boldsymbol{\alpha}$}
\psfrag{gb}{$\boldsymbol{\beta}$}
\psfrag{gc}{$\boldsymbol{\gamma}$}
\psfrag{aa}{$A=a+a^{-1}$}
\psfrag{bb}{$B=b+b^{-1}$}
\psfrag{cc}{$C=c+c^{-1}$}
\psfrag{dd}{}
\includegraphics[width=12cm]{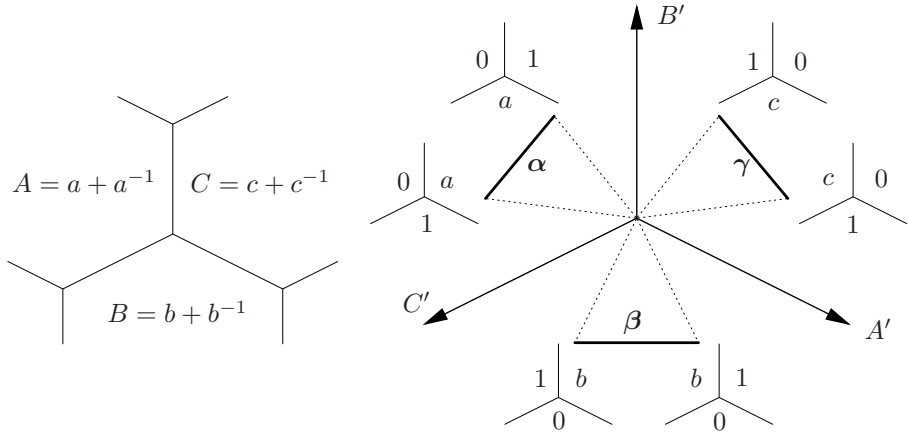}
\caption{The derivative of a Markoff map.} \label{fig:markoff} 
\end{figure}

Figure \ref{fig:markoff} shows the segment $\boldsymbol{\alpha}$ together with the other two segments $\boldsymbol{\beta}, \boldsymbol{\gamma}$ obtained by exchanging the roles of regions $R_{\infty},R_0,R_1$ cyclically. The notation is $$\Phi_t(R_{\infty})=A=a+a^{-1}~;~ \Phi_t(R_0)=B=b+b^{-1}~;~ \Phi_t(R_1)=C=c+c^{-1}~,$$ agreeing that $a,b,c>1$; and derivation with respect to $t$ is denoted by $'$. The left panel shows the original Markoff map $\Phi=\Phi_0$, more precisely $\Phi_0(R)=A,B,C$ written in regions $R=R_\infty, R_0, R_1$. The right panel shows the derivative of $\Phi$ (the numbers $\Phi'(R)=\frac{\text{d}}{\text{d}t}\Phi_t(R)=A', B', C'$ written inside regions $R$) in six different directions $\Phi'$ of 3-dimensional space $\mathbb{R}^3$. These 6 directions project in $\mathbb{P}^2\mathbb{R}$ to the endpoints of segments $\boldsymbol{\alpha}, \boldsymbol{\beta},\boldsymbol{\gamma}$. Namely, $\boldsymbol{\beta}$ connects $[1:0:b]$ to $[b:0:1]$ and $\boldsymbol{\gamma}$ connects $[1:c:0]$ to $[c:1:
0]$. The key observation, which is obvious from the diagram (or the coordinates), is that after projectivization \emph{the six directions form a convex hexagon in $\mathbb{P}^2\mathbb{R}$}.

Note also that $\boldsymbol{\alpha}, \boldsymbol{\beta}, \boldsymbol{\gamma}$ are all in the nonnegative octant. To conclude that the cone of deformations which lengthen all curves is nonempty and has one top-dimensional face for each slope, it is enough to make sure that the constraints $\Phi'(R)>0$ (for $R$ ranging over all regions) will only chop off the \emph{corners} of the projective triangle defined by the non-negative octant, never affecting the segments $\boldsymbol{\alpha}, \boldsymbol{\beta}, \boldsymbol{\gamma}$ themselves. In fact, it is enough to do this for $R$ equal to the other common neighbor $R_2$ ($\neq R_0$) of $R_\infty$ and $R_1$, and then use induction. This is accomplished in Section \ref{beeg}, and illustrated in Figure \ref{fig:markoff-bis}. 

\subsection{The infinite polygon of lengthening deformations} \label{beeg}

\begin{figure}[h!] \centering
\psfrag{un}{\sc i}
\psfrag{2}{\sc ii}
\psfrag{3}{\sc iii}
\psfrag{4}{\sc iv}
\psfrag{5}{\sc v}
\psfrag{6}{\sc vi}
\psfrag{7}{\sc vii}
\psfrag{8}{\sc viii}
\psfrag{9}{\sc ix}
\psfrag{10}{\sc x}
\psfrag{11}{\sc xi}
\psfrag{12}{\sc xii}
\psfrag{a}{$a$}
\psfrag{b}{$b$}
\psfrag{c}{$c$}
\psfrag{d}{$d$}
\psfrag{ga}{$\boldsymbol{\alpha}$}
\psfrag{gb}{$\boldsymbol{\beta}$}
\psfrag{gc}{$\boldsymbol{\gamma}$}
\psfrag{gd}{$\boldsymbol{\delta}$}
\psfrag{aa}{$A=a+a^{-1}$}
\psfrag{bb}{$B=b+b^{-1}$}
\psfrag{cc}{$C=c+c^{-1}$}
\psfrag{dd}{$D=d+d^{-1}$}
\psfrag{0}{$0$}
\psfrag{1}{$1$}
\psfrag{am}{$a^2$}
\psfrag{cm}{$c^2$}
\psfrag{ap}{$a^2+1$}
\psfrag{cp}{$c^2+1$}
\psfrag{dapa}{$\begin{array}{c} c+c^{-1} \\ +d(a+a^{-1}) \end{array}$}
\psfrag{dcpc}{$\begin{array}{c} a+a^{-1} \\ +d(c+c^{-1}) \end{array}$}
\psfrag{bapa}{$\begin{array}{c} b(a+a^{-1}) \\ +c+c^{-1} \\ \end{array}$}
\psfrag{bcpc}{$\begin{array}{c} b(c+c^{-1}) \\ +a+a^{-1} \\ \end{array}$}
\includegraphics[width=13cm]{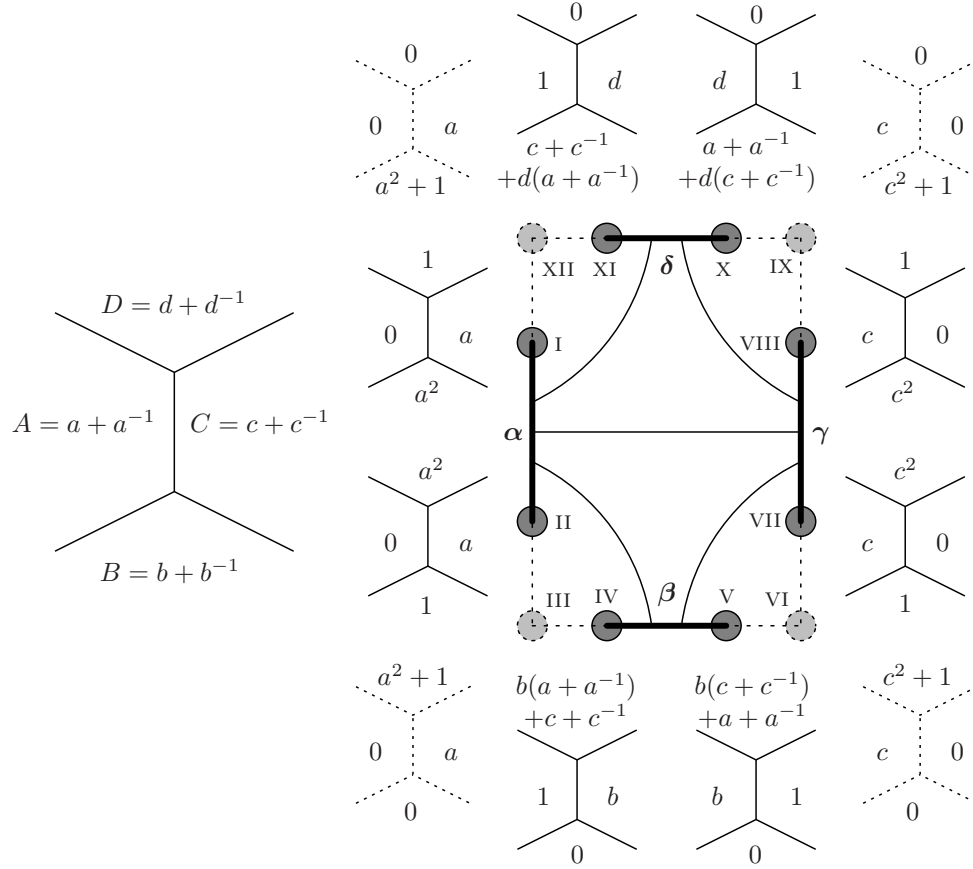}
\caption{The derivative of a Markoff map.} \label{fig:markoff-bis} 
\end{figure}

Figure \ref{fig:markoff-bis} uses the same conventions as Figure \ref{fig:markoff}, but shows a little more. The left panel shows part of a Markoff map (four regions). The right panel shows its derivatives in twelve different (projective) directions, numbered {\sc i} through {\sc xii}. Segment $\boldsymbol{\alpha}$ from Figure \ref{fig:markoff} corresponds to {\sc i--ii}; $\boldsymbol{\beta}$ is {\sc iv--v}; $\boldsymbol{\gamma}$ is {\sc vii--viii}. There is a new segment $\boldsymbol{\delta}=\text {\sc x--xi}$ whose projective coordinates are also computed. The line extending $\boldsymbol{\delta}$ chops a corner off the nonnegative octant circumscribed to $\boldsymbol{\alpha}, \boldsymbol{\beta}, \boldsymbol{\gamma}$, away from $\boldsymbol{\alpha}\cup \boldsymbol{\beta}\cup \boldsymbol{\gamma}$ and opposite $\boldsymbol{\beta}$. The lines extending $\boldsymbol{\alpha}, \boldsymbol{\beta}, \boldsymbol{\gamma}, \boldsymbol{\delta}$ meet at extra points {\sc iii--vi--ix--xii}, whose coordinates are also 
shown (computing them is an easy exercise), and which span a convex quadrilateral $Q$  in $\mathbb{P}^2\mathbb{R}$. The sides of $Q$ contain $\boldsymbol{\alpha}, \boldsymbol{\beta}, \boldsymbol{\gamma}, \boldsymbol{\delta}$ in their interiors because $a,b,c,d>1$.

We can repeat the argument inductively for all regions $R$: later lines defining the infinite polygon $\Pi$ chop off only corners, never affecting earlier segments $\boldsymbol{\alpha}, \boldsymbol{\beta},\dots$. The argument is the same at each stage, up to a projective transformation of the plane: see Figure \ref{fig:induct}.

This induction proves

\begin{theorem} \label{thm:main}
If $\Phi$ is a Markoff map satisfying \eqref{eqii}, then the cone of lengthening deformations projectivizes to a projective polygon $\Pi$ contained in $Q$ and containing the convex hull of $\boldsymbol{\alpha}\cup \boldsymbol{\beta}\cup \boldsymbol{\gamma}\cup \boldsymbol{\delta}$. Moreover, $\Pi$ has one nondegenerate side for each rational slope; for example the sides $\boldsymbol{\alpha}, \boldsymbol{\beta}, \boldsymbol{\gamma}, \boldsymbol{\delta}$ correspond to the slopes attached to the regions carrying $A,B,C,D$.
\end{theorem}

\begin{figure}[h!] \centering
\psfrag{a}{$\boldsymbol{\alpha}$}
\psfrag{b}{$\boldsymbol{\beta}$}
\psfrag{c}{$\boldsymbol{\gamma}$}
\psfrag{d}{$\boldsymbol{\delta}$}
\includegraphics[width=4.5cm]{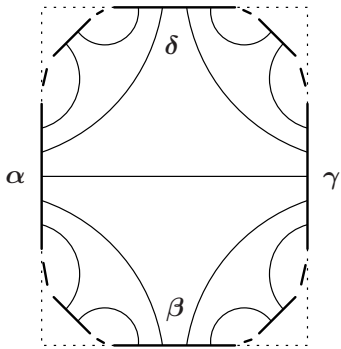}
\caption{Illustration of the induction proving Theorem \ref{thm:main}. (Curved arcs inside the infinite polygon $\Pi$ are there only to indicate the Farey combinatorics.)} \label{fig:induct} 
\end{figure}

Theorem \ref{thm:main-intro} is equivalent to Theorem \ref{thm:main}, as $\mathcal{T}$ identifies with the space of Markoff maps satisfying \eqref{eqii}.

In Sections \ref{sec:onecurve}--\ref{sec:allcurves}, we will detail how Theorem \ref{thm:main-intro} can degenerate: 
\begin{itemize} \item First, when $\Phi(R)$ becomes equal to $2$ for one isolated region $R$ (Theorem~\ref{thm:onepinch}): the corresponding side of $\Pi$ collapses to a point at which $\partial \Pi$ is pinched between two conics with a contact of highest possible order, $4$.
\item  Then, when $\Phi(R)=2$ for all regions $R$ simultaneously (Theorem \ref{thm:allpinch}): $\partial \Pi$ becomes a conic itself. 
\end{itemize}
But first we make some explicit computations.

\section{Computations} \label{sec:computation}

Let $\Phi:\mathcal{R}\rightarrow (2,+\infty)$ be a geometric Markoff map, i.e.\ satisfy \eqref{eqii}. In this section we give some quantitative estimates for the shape of the associated infinite polygon $\Pi$. More precisely, we shall in \eqref{eq:endpoints} give explicit formulas for the relative positions in $\mathbb{P}^2\mathbb{R}$ of an edge $E_R=E_{R_\infty}$ of $\Pi$ and all its Farey neighbors $(E_{R_n})_{n\in \mathbb{Z}}$.

Section \ref{sec:asymptotics} will then deal with asymptotics: for example, we will show that near the ends of $E_R$, ``most'' of the perimeter $\partial \Pi$ is taken up by the union of the edges $E_{R_n}$, leaving only very small interstices for the $E_{R_s}$ with $s\notin\mathbb{Z}$ (Remark \ref{interstices}).

Consider a region $R=R_{\infty}$ of the Markoff map $\Phi>2$, and the family of adjacent regions $(R_n)_{n\in \mathbb{Z}}$. by Proposition \ref{prop:commutator} there exist reals $\ell>0, y>1$ and $x$ such that 
\begin{equation} \label{eq:zo}
\begin{array}{rcl}
\Phi(R)&=& 2\cdot \cosh \ell \\ 
\Phi(R_n)&=& 2y\cdot \cosh (n\ell-x)~\text{ for all $n\in \mathbb{Z}$.}
\end{array}
\end{equation}

Let the parameters $\ell, y, x$ depend on a real variable $t$ and suppose $$\frac{\text{d}}{\text{d}t}(\ell, y, x)=(L,Y,X)~.$$ Then 
\begin{equation} \label{eq:zozo}
\begin{array}{llcl}
\Phi'(R) & :=  \frac{\text{d}}{\text{d}t}\Phi(R) &=& 2L \sinh \ell \\ 
\Phi'(R_n) & :=  \frac{\text{d}}{\text{d}t}\Phi(R_n) &=& 2Y \cosh (n\ell-x) + 2y(nL-X)\sinh (n\ell-x)~.
\end{array}
\end{equation}
We will use $L,Y,X$ as coordinates in the space of infinitesimal deformations of $\Phi$.

\subsection{The main edge}
Apply Remark \ref{rem:geoseq}: the two endpoints $P^+, P^-$ of the side $E_R$ of the projective polygon $\Pi$ corresponding to $R$ are the projective triples $[L^{\pm}:Y^{\pm}:X^{\pm}]$ such that 
$$ \left \{  \begin{array}{l} 
\Phi'(R) =0~\text{ (i.e.\ $L=0$), and} \\ 
\left ( \Phi'(R_n)\right )_{n\in \mathbb{Z}} ~\text{ is a geometric progression.}
\end{array} \right .$$
In other words, $P^{\pm}=[L^{\pm}:Y^{\pm}:X^{\pm}]=[0,y,\pm 1]$.

\subsection{The neighboring edges} \label{sec:neighboring}
Similarly, by Remark \ref{rem:geoseq}, the two endpoints $P_n^+, P_n^-$ of the side $E_{R_n}$ of the projective polygon $\Pi$ corresponding to $R_n$ are the projective triples $[L:Y:X]=[L_n^{\pm}:Y_n^{\pm}:X_n^{\pm}]$ such that 
$$ \left \{  \begin{array}{l} 
\Phi'(R_n) =0~\text{ and} \\ 
\Phi'(R_{n+1})~,~ \Phi'(R)~,~ \Phi'(R_{n-1}) ~\text{ is a geometric progression.}
\end{array} \right .$$
By Theorem \ref{thm:main}, the edge $E_{R_n}$ lies entirely on one side of $E_R$ in any affine chart of the convex polygon $\Pi$: therefore we can normalize to $L=1$ when solving the above identities. 
The first identity $\Phi'(R_n)=0$ can then be written, using \eqref{eq:zozo},
$$X=\frac{Y}{y\cdot \tanh (n\ell-x)}+n$$
which will allow us to replace $X$ or $n-X$ with its value in the sequel. If we abbreviate $n\ell-x$ as $\xi$, then the second identity, $\Phi'^2(R)=\Phi'(R_{n+1})\Phi'(R_{n-1})$, can be written (for $L=1$)
\begin{eqnarray*} 
\sinh^2 \ell 
&=&
[Y\cdot \cosh (n\ell-x+\ell) + y(n-X+1)\sinh (n\ell-x+\ell)] 
\\ && 
\times 
[Y\cdot \cosh (n\ell-x-\ell) + y(n-X-1)\sinh (n\ell-x-\ell)] \\
&=& \textstyle{\left [Y\, \cosh (\xi+\ell) - \left ( \frac{Y}{\tanh \xi}-y \right )\sinh (\xi+\ell)\right ]} 
\\ && 
\times  
\textstyle{\left [Y\, \cosh (\xi-\ell) - \left ( \frac{Y}{\tanh \xi}+y \right )\sinh (\xi-\ell)\right]} 
\\ 
&=& 
Y^2\cosh(\xi+\ell)\cosh(\xi-\ell) + \left [ \frac{Y^2}{\tanh^2 \xi}-y^2\right ] \sinh(\xi+\ell)\sinh(\xi-\ell) 
\\ && 
\hspace{20pt} 
-\frac{Y^2}{\tanh \xi} [\cosh (\xi+\ell) \sinh (\xi-\ell)+\cosh (\xi-\ell)\sinh(\xi+\ell)] 
\\ && 
\hspace{20pt}\hspace{20pt} 
+Yy [\cosh (\xi-\ell)\sinh(\xi+\ell)-\cosh(\xi+\ell)\sinh(\xi-\ell)] \\
&=& Y^2 \left ( \cosh^2 \xi + \sinh^2 \ell + \frac{\sinh^2 \xi-\sinh^2 \ell}{\tanh^2 \xi} - \frac{\sinh 2 \xi}{\tanh \xi} \right ) 
\\ && 
\hspace{20pt} 
+Yy\sinh 2\ell + y^2 (\sinh^2 \ell-\sinh^2 \xi) 
\\
&=& -Y^2 \frac{\sinh^2 \ell}{\sinh^2 \xi} +2Yy\sinh \ell \cosh \ell + y^2 (\sinh^2 \ell-\sinh^2 \xi)~;
\end{eqnarray*}
hence $\displaystyle{Z:=\frac{Y}{y}}$ is a root of the polynomial
$$ Z^2 \frac{\sinh^2 \ell}{\sinh^2 \xi} - 2Z \sinh \ell \cosh \ell + \left [\sinh^2 \xi -(1-y^{-2})\sinh^2 \ell \right ]$$
with discriminant 
\begin{eqnarray*} \Delta &:=& (\sinh \ell \cosh \ell)^2 - \sinh^2 \ell + (1-y^{-2})\frac{\sinh^4 \ell}{\sinh^2 \xi} \\ &=& \frac{\sinh^4 \ell}{\sinh^2 \xi} \cdot ( \cosh^2 \xi -y^{-2})~.
\end{eqnarray*}
Recall $y>1$ by Proposition \ref{prop:commutator}, so $\Delta>0$.
We thus find the two projective triples $P_n^{\pm}=[L_n^{\pm}:Y_n^{\pm}:X_n^{\pm}]$ given by
\begin{eqnarray}
L_n^{\pm}&=&1 \notag \\
Y_n^{\pm}&=& \frac{y\cdot \sinh ^2 \xi_n}{\tanh \ell} \pm y\cdot \sinh \xi_n \sqrt{\cosh^2 \xi_n -y^{-2}} \label{eq:endpoints} \\
X_n^{\pm}&=& \frac{\sinh \xi_n \cosh \xi_n}{\tanh \ell} \pm \cosh \xi_n \sqrt{\cosh^2 \xi_n -y^{-2}} + n \notag
\end{eqnarray}
(using the shorthand $\xi_n=n\ell-x=\xi$ as before). Do not overlook the ``$+n$'' at the end of the expression of $X_n^{\pm}$.

\subsection{Limits near $\pm\infty$}
Since $\cosh \xi \sim \pm \sinh\xi$ at $\pm\infty$ and $\frac{1}{\tanh \ell} \pm 1 = \frac{e^{\pm \ell}}{\sinh \ell}$, the limits of $P_n^{\pm}$ as $n$ (and therefore~$\xi$) goes to $+\infty$ and $-\infty$ are computed as follows:
\begin{eqnarray*}
\left [ L_{+\infty}: Y_{+\infty}: X_{+\infty} \right ] &:=& \left [0:y \frac{e^{\pm\ell}}{\sinh \ell}: \frac{ e^{\pm\ell}}{\sinh \ell} \right ] = [0:y:1] = P^+\\
\left [ L_{-\infty}: Y_{-\infty}: X_{-\infty} \right ] &:=& \left [0:y \frac{e^{\pm\ell}}{\sinh \ell}: \frac{-e^{\pm\ell}}{\sinh \ell} \right ] = [0:y:-1]= P^-~.
\end{eqnarray*}
As expected, the sides $E_{R_n}$ of the infinite polygon $\Pi$ accumulate near the two endpoints $P^{\pm}$ of $E_R$ as $n$ goes to $+\infty$ or $-\infty$. See Figure \ref{fig:parabole} for a representation of the polygon $\Pi$ in the $(X,Y)$-plane.

\begin{figure}[h!] \centering
\psfrag{A}{$Y$}
\psfrag{B}{$X$}
\psfrag{Pmp}{$P_{-1}^+$}
\psfrag{Pm}{$P^+$}
\psfrag{Pp}{$P^-$}
\psfrag{P0m}{$P_0^-$}
\psfrag{P0p}{$P_0^+$}
\psfrag{P1m}{$P_1^-$}
\psfrag{P1p}{$P_1^+$}
\psfrag{P2m}{$P_2^-$}
\psfrag{P2p}{$P_2^+$}
\psfrag{P3m}{$P_3^-$}
\psfrag{Pi}{$\Pi$}
\includegraphics[width=8cm]{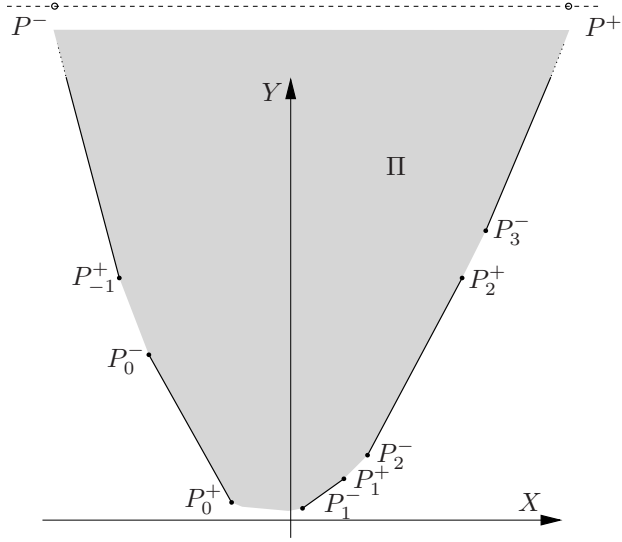}
\caption{A few sides $E_{R_n}=[P_n^-P_n^+]$ of the infinite polygon $\Pi$. The side $E_R=[P^-P^+]$ is at infinity (dotted). This is the affine chart $L=1$.} \label{fig:parabole} 
\end{figure}

\section{Asymptotics} \label{sec:asymptotics}

Let us now estimate the shape of the polygon $\Pi$, focussing on values of $n$ near $+\infty$ (the case of $-\infty$ is similar). Define
$$ \Xi_n:=e^{\xi_n}=e^{n\ell-x}~ \text{ for all }~ n\in \mathbb{Z},$$ so that $(\Xi_n)_{n\in\mathbb{Z}}$ is a geometric sequence of ratio $e^\ell$. Note that $$\sqrt{1+u z + v z^2}=1+\frac{u}{2}z+\left (\frac{v}{2}-\frac{u^2}{8} \right ) z^2+O(z^3)$$ for all reals $u,v$: therefore 
\begin{eqnarray*}
\sqrt{\cosh^2 \xi_n -y^{-2}}
&=& \frac{1}{2}\sqrt{\Xi_n^2+(2-4y^{-2})+\Xi_n^{-2}} \\
&=& \frac{\Xi_n}{2} + (1-2y^{-2})\frac{\Xi_n^{-1}}{2} + (2y^{-2}-2y^{-4})\frac{\Xi_n^{-3}}{2} +O(\Xi_n^{-5})
\end{eqnarray*}
and finally (normalizing by $4$ and $\frac{4}{y}$ to make the right-hand sides more pleasant)
\begin{eqnarray*}
\frac{4}{y}Y_n^+&=&\frac{e^{\ell}}{\sinh \ell}\Xi_n^2 -\left ( \frac{2}{\tanh \ell}+2y^{-2} \right ) + \left ( \frac{e^{-\ell}}{\sinh \ell} +4y^{-2}-2y^{-4}\right )\Xi_n^{-2}+O(\Xi_n^{-4})
\\ 
\frac{4}{y}Y_n^-&=& \frac{e^{-\ell}}{\sinh \ell}\Xi_n^2 - \left (\frac{2}{\tanh \ell}-2y^{-2}\right ) + \left ( \frac{e^{\ell}}{\sinh \ell} -4y^{-2}+2y^{-4}\right )\Xi_n^{-2}+O(\Xi_n^{-4})
\\ 
4X_n^+&=& \frac{e^{\ell}}{\sinh \ell}\Xi_n^2 + \left ( 2-2y^{-2} \right ) + \left ( \frac{-e^{-\ell}}{\sinh \ell} -2 y^{-4} \right ) \Xi_n^{-2}+4n+O(\Xi_n^{-4}) 
\\
4X_n^-&=& \frac{e^{-\ell}}{\sinh \ell}\Xi_n^2 - \left ( 2 - 2y^{-2} \right ) + \left ( \frac{-e^{\ell}}{\sinh \ell} +2 y^{-4} \right ) \Xi_n^{-2}+4n+O(\Xi_n^{-4})~.
\end{eqnarray*}

\subsection{Sides fill up} \label{edgesfillup}
This allows us to compute the side and chord vectors of $\Pi$: if $P_n^{\pm}=(X_n^{\pm}, Y_n^{\pm})$, then the vector directing a side $E_{R_n}=[P_n^-, P_n^+]$ of $\Pi$ has coordinates $(X_n^+-X_n^-)$ and $(Y_n^+-Y_n^-)$ given by
\begin{eqnarray*}
\frac{4}{y}(Y_n^+-Y_n^-)&=& 2\Xi_n^2 -4y^{-2} + O(\Xi_n^{-2})
\\ 4(X_n^+-X_n^-)&=& 2\Xi_n^2 +4 - 4 y^{-2} + O(\Xi_n^{-2}).
\end{eqnarray*}
The vector directing a chord $P_n^+ P_{n+1}^-$ has coordinates given by
\begin{eqnarray*}
\frac{4}{y}(Y_{n+1}^- -Y_n^+)&=& 4 y^{-2} - (4y^{-2}-2y^{-4})(1+e^{-2\ell})\Xi_n^{-2} +O (\Xi_n^{-4})
\\ 4(X_{n+1}^--X_n^+)&=& 4 y^{-2} + 2y^{-4}(1+e^{-2\ell})\Xi_n^{-2} + O(\Xi_n^{-4}).
\end{eqnarray*}
Note that for large $n$ the vertex $P_n^+$ is roughly $$\frac{y^2\Xi_n^2}{2}\gg 1$$ times closer to $P_{n+1}^-$ than to $P_n^-$. Therefore,
\begin{remark} \label{interstices} The proportion of the perimeter of $\Pi$ filled up by the edges $E_{R_n}=[P_n^-, P_n^+]$ grows asymptotically to full measure as $n$ goes to infinity.\end{remark}

\begin{observation}
With some work, one can very likely prove that the complement in $\partial \Pi$ of the union $\displaystyle{\bigcup_{\frac{p}{q} \in \mathbb{P}^1\mathbb{Q}} E_{R_{\frac{p}{q}}}}$ has Hausdorff dimension $0$. See \cite[\S 10]{stretchmaps}, \cite{glmm}.
\end{observation}

\subsection{Slopes}
We can also compute the slope of a side $P_n^- P_n^+$:
$$
\sigma_n:=\frac{Y_n^+-Y_n^-}{X_n^+-X_n^-} = y \tanh \xi_n = y(1- 2 \Xi_n^{-2})+O(\Xi_n^{-4})
$$
and the slope of a chord $P_n^+ P_{n+1}^-$:
$$
\sigma'_n:=\frac{Y_{n+1}^--Y_n^+}{X_{n+1}^--X_n^+} = y\left (1-(1+e^{-2\ell}) \Xi_n^{-2}\right )+O(\Xi_n^{-4})~.
$$
As expected by convexity of $\Pi$, the sequence $$\sigma_n~,~ \sigma'_n~,~ \sigma_{n+1}~,~ \sigma'_{n+1}~,~ \sigma_{n+2}~,~\sigma'_{n+2}~,~ \dots$$ is monotonically increasing for large $n$ since $\Xi_{n+1}=e^{\ell}\Xi_n$ and $2>1+e^{-2\ell}>2e^{-2\ell}$.

\subsection{Asymptotes}
Notice finally that the perimeter of $\Pi$ has no asymptote (although it has an asymptotic direction) as $n$ goes to $+ \infty$ or $-\infty$, which means the endpoints of the projective side at infinity $[P^-P^+]$ are smooth. To see this, it is enough to compute from (\ref{eq:endpoints}) that the line extending the side $[P_n^- P_n^+]$ intersects the horizontal axis at the coordinate $$\frac{Y_n^+ X_n^- - Y_n^- X_n^+}{Y_n^+ - Y_n^-}=n$$ which diverges (pleonastically!) for $n \rightarrow \pm \infty$.

\section{Limiting case: pinching one curve} \label{sec:onecurve}
We now consider the limit case of \eqref{eqii} where $\Phi\geq 2$ is a Markoff map with just one entry equal to $2$, and seek infinitesimal deformations of $\Phi$ such that all entries increase. In \eqref{eq:zo} this can be seen as the limit case $\ell=0$.

Since the derivative of $\cosh$ at $0$ is $0$, the above parameterization of the tangent space to Markoff maps by $L,Y,X$ is no longer valid. Instead, we have $\Phi(R)=2\cosh \ell = 2$, hence from \eqref{eq:markoff} the relationship $\Phi(R_{n+1})+\Phi(R_{n-1})=2\Phi(R_n)$ which yields $$\Phi(R_n)=zn+2y$$ (analogue of \eqref{eq:zo}) for certain real constants $z,y$. Necessarily, $z=0$ and $y\geq 1$ because $\Phi\geq 2$. In this section we focus on $$y>1.$$

By \eqref{traceid}, if $\Phi(R)=2$, then the trace $K$ of the commutator is $2+(\Phi(R_{n+1})-\Phi(R_{n}))^2=2$: in particular, we are automatically also in the degenerate case of \eqref{eqii} where $K=2$ (cf Remark \ref{rem:linked}). This means the singular torus $S$ has collapsed to a circle, of length $2\, \mathrm{arccosh}(y)>0$ (the common length of all the loops that intersect the collapsing loop exactly once).

Keeping $\Phi(R)=2$ constant, the only length-expanding deformation is therefore given by stretching the collapsed circle, i.e.\ (up to scaling) $\text{d}y/\text{d}t=1$. This means the $R$-edge of the polygon $\Pi$ has collapsed to a single point $r$. Let us now study how the $R_n$-edges of $\Pi$ accumulate to $r$.

\begin{theorem} \label{thm:onepinch}
All the $R_n$-edges (for $n\in\mathbb{Z}$) have their endpoints on a conic through $r\in\mathbb{P}^2\mathbb{R}$. Moreover, all these edges are tangent to a second conic, which intersects the first one at $r$ with a contact of order 4. 
\end{theorem}
\begin{proof}
Differentiating \eqref{eq:markoff}, all deformations $\Phi'$ normalized to $\Phi'(R)=1$ satisfy 
$$ \Phi'(R_{n+1})+\Phi'(R_{n-1})=2y+2\Phi'(R_n)$$ 
hence $$\Phi'(R_n)=yn^2-Xn+Y$$ for certain constants $X,Y$ depending only on $\Phi'$ (this is the analogue\footnote{Although the roles of variables $y, Y$ are analogous to \S 3, we cannot here interpret $Y$ as a ``derivative''~$y'$: in fact the notation $y'$ does not a priori make sense when $\Phi'(R)\neq 0$.} of \eqref{eq:zozo}). As in Section \ref{sec:neighboring} above, by Remark \ref{rem:geoseq} the endpoints of the $R_n$-edge are the pairs $(X,Y)=(X_n^\pm, Y_n^\pm)$ such that 
\begin{eqnarray*}\Phi'(R_n)&=&0 ~\text{ and } \\ \Phi'(R_{n+1})\Phi'(R_{n-1})&=&\Phi'^2(R)=1\end{eqnarray*}
which can be written
\begin{eqnarray*}
 0&=& yn^2-Xn+Y \\
1 &=& [y(n+1)^2-X(n+1)+Y] \cdot [y(n-1)^2-X(n-1)+Y] \\
  &=& (2ny+y-X)(-2ny+y+X) = y^2 - (X-2ny)^2~.
\end{eqnarray*}
This admits the solutions 
$$
\begin{array}{rcl}
X=X_n^{\pm}&=&2ny\pm\sqrt{y^2-1} \\
Y=Y_n^{\pm}=X_n^{\pm} n-yn^2 &=& yn^2\pm n\sqrt{y^2-1}~.
\end{array}
$$
If we define $f(X)=\frac{X^2}{4y}$ and the arithmetic progression $X_n=2y\cdot n$, we therefore have $f(X_n)=n^2y$ and $f'(X_n)=n$, hence 
$$\begin{array}{rcr@{X_n}l@{~\pm~}l}
X_n^{\pm}&=&&&\sqrt{y^2-1} \\
Y_n^{\pm}&=&f(&)& \sqrt{y^2-1} \cdot f'(X_n)~.
\end{array}
$$
Therefore the edges $E_{R_n}$ of the polygon $\Pi$, for $n\in\mathbb{Z}$,  are all tangent to a common parabola (the graph of $f$). The endpoints of the edges $E_{R_n}$ lie on a second parabola, a translate of the first one along its asymptotic axis, downwards by $f(\sqrt{y^2-1})=\frac{y-y^{-1}}{4}$ units of length. Algebraically, the contact between these two parabolas at infinity is of order $4$. This proves Theorem \ref{thm:onepinch}. \end{proof}

The expression of $X_n^\pm$ and $Y_n^\pm$ in the proof above shows that the union of the edges $E_{R_n}$ occupies an asymptotic proportion $\sqrt{y^2-1}/y<1$ of the perimeter of $\Pi$ near the point $r$ (as $n$ goes to $\pm \infty$). This is in contrast with the behavior for geometric $\Phi$ (Remark \ref{interstices}). Qualitatively, $\Pi$ looks similar to
Figure \ref{fig:parabole}, except that the two parabolic branches have the same asymptotic (vertical) direction, so that $P^+=P^-$ in $\mathbb{P}^2\mathbb{R}$.

\section{Limiting case: pinching all curves} \label{sec:allcurves}

Suppose finally that $\Phi(R)=2$ for all regions $R$ of the Markoff map --- or equivalently, that $y=1$ in the above. Then the corresponding representation is (unipotent or) trivial. An infinitesimal, length-increasing deformation will correspond to a singular hyperbolic torus of infinitesimally small diameter, i.e.\ a \emph{Euclidean} torus. The space of Euclidean tori (of normalized area, say) is canonically identified with $\mathbb{H}^2$, a round projective disk; and we will show that indeed,
\begin{theorem} \label{thm:allpinch}
When $\Phi\equiv 2$, the polygon $\Pi$ degenerates to a round projective disk which is also the tangent cone, at $(2,2,2)$, of the closure of the deformation space $\mathcal{X}$ of Section~\ref{curlix}.
\end{theorem}
\begin{proof}
By \eqref{eq:markoff}, for every $4$-tuple of distinct regions $R,R',R'',R'''$ such that all pairs except $\{R,R'''\}$ are adjacent, a deformation of the constant Markoff map $\Phi\equiv 2$ must satisfy 
\begin{equation} \label{round}
 \Phi'(R)+\Phi'(R''')=2(\Phi'(R')+\Phi'(R''))~.
\end{equation}

Consider the hyperbolic plane $\mathbb{H}^2$ bounded by $\mathbb{P}^1\mathbb{R}$ and endowed with the Farey triangulation and a $\text{PSL}_2(\mathbb{Z})$-invariant family of horoballs $(H_q)$ centered at the rationals $q\in\mathbb{Q}\cup\{\infty\}$. Above every $H_q$ lives a lightlike vector $v_q$ in Minkowski space $\mathbb{R}^{2,1}$ such that $H_q$ is the intersection of the semi-hyperboloid $\mathbb{H}^2$ with the affine half-space $\{v\in\mathbb{R}^{2,1}~|~\langle v |v_q\rangle \leq 1\}$, where $\langle (x,y,x)|(x',y',z')\rangle = xx'-yy'-zz'$ is the Minkowski product.

The pairing $\langle v_q | v_{q'}\rangle$ depends only on the signed hyperbolic distance $d$ between $H_q$ and $H_{q'}$ (it is $2e^d$ though the actual function does not matter) and is therefore the same for any pair of Farey neighbors $q,q'$, by $\text{PSL}_2(\mathbb{Z})$-invariance. Hence, if $q,q',q'',q'''$ are distinct rationals in $\mathbb{P}^1\mathbb{Q}$ such that every pair except $(q,q''')$ consists of Farey neighbors, it is easy to see that up to action by an element of $\text{O}_{2,1}(\mathbb{R})$,
$$\begin{array}{rc@{~(}r@,r@,r@)}
v_q&=&2 &2 &0  \\
v_{q'''}&=&2 &-2& 0\\
v_{q'}&=&1& 0& 1\\
v_{q''}&=&1 &0 &-1
\end{array}$$
hence
$$ v_q+v_{q'''}=2(v_{q'}+v_{q''})~.$$
This is formally the same identity as (\ref{round}).
Therefore, and by an immediate induction on the Farey graph, one has $\Phi'(R_q)=\eta(v_q)$ for an appropriate linear form $\eta$ depending only on the infinitesimal deformation $\Phi'$ (not on $q$). Dually, we can view $\eta$ itself as a vector of $\mathbb{R}^{2,1}$. But the $[v_q]$ are dense in the projectivized isotropic cone, because $\mathbb{Q}$ is dense in $\mathbb{R}$. The condition ``$\Phi'(R_q)>0$ for all $q$'' on $\eta$ therefore defines a round cone in $\mathbb{R}^{2,1}$ (minus the rational rays on its boundary), as claimed. Theorem \ref{thm:allpinch} is proved, as Remark \ref{rem:linked} also shows the lengthening condition defines a round cone of directions from the tip $(2,2,2)$ of $\mathcal{X}$. \end{proof}

Let us finally show that the interior of this cone of deformations of $\Phi\equiv 2$ is naturally identified with the space of flat Euclidean structures on the marked torus, which is also the space of Euclidean quadrilaterals up to similarity. 

Recall that every value of a geometric Markoff map is (twice) the hyperbolic cosine of the (half) length of a geodesic in the hyperbolic punctured torus $S$. 
Since $\cosh \ell=1+\frac{1}{2}\ell^2+o(\ell^3)$ as $\ell \rightarrow 0$, for a ``lengthening'' deformation $\Phi'$ of the nongeometric Markoff map $\Phi\equiv 2$, we should interpret $\Phi'(R_q)$ as the \emph{square} of the renormalized length of a curve of slope $q\in \mathbb{P}^1\mathbb{Q}$ in an infinitesimally small hyperbolic (i.e.\ Euclidean) torus.   

More precisely, if $\gamma, \gamma', \gamma'', \gamma'''$ are geodesic loops in a flat torus such that each pair except $(\gamma, \gamma''')$ intersects exactly once, then their \emph{squared} lengths $l,l',l'',l'''$ satisfy $$l+{l'''}=2(l'+l'')$$
hence again by an immediate induction $l_q=\eta(v_q)$ for some linear form $\eta$ (determined by $(l_0, l_1, l_\infty)$ and independent of $q$), where $l_q$ is the squared length of the loop of slope $q$. The condition that $\sqrt{l_0},\sqrt{l_1},\sqrt{l_\infty}$ satisfy the strong triangle inequality (i.e.\ are realized in some Euclidean metric) can be written
$$l_0^2+l_1^2+l_\infty^2-2(l_0 l_1 + l_1 l_\infty + l_\infty l_0)<0$$
which again defines a cone on a round disk, tangent to all three faces of the projectivized positive octant. This is the same cone as the one given by the condition that $\eta$ be positive on all $v_q$. It is also (after projectivization) the round disk of Figure \ref{fig:wedge}.

Thus the algebraic action of the mapping class group $\mathrm{SL}_2\mathbb{Z}$ on (the closure of) the character manifold $\mathcal{X}$ linearizes, near the fixed point $(2,2,2)$ representing the trivial representation, to the usual Lorentzian action on $\mathbb{R}^{2,1}$ given by the irreducible 3-dimensional representation of $\mathrm{SL}_2\mathbb{R}$. It would be interesting to know how, or how far, this picture extends to surfaces of higher complexity.

\begin{flushright}
\end{flushright}

\end{document}